\documentclass[reqno]{amsart}
\usepackage[utf8]{inputenc}
\usepackage[english]{babel}
\usepackage{amstext}
\usepackage{latexsym}
\usepackage{amsmath}
\usepackage{amssymb}
\usepackage{amsfonts}
\usepackage{comment}
\usepackage{graphicx}
\usepackage{amssymb}
\usepackage{amsthm}
\usepackage{epsfig}
\usepackage{tikz}
\usetikzlibrary{matrix,arrows,decorations.pathmorphing}
\usepackage{graphicx}
\usepackage{enumitem}
\usepackage{xcolor}
\allowdisplaybreaks

\usepackage[hidelinks,colorlinks=true,citecolor={red},linkcolor= black]{hyperref}

\input xy
\xyoption{all}
\newtheorem{defn}{Definition}[section]
\newtheorem{coro}[defn]{Corollary}
\newtheorem{lemma}[defn]{Lemma}
\newtheorem{rem}[defn]{Remark}
\newtheorem{prop}[defn]{Proposition}
\newtheorem{theo}[defn]{Theorem}

\newtheorem{claim}[defn]{Claim}
\newtheorem{ltheorem}{Theorem}

\begin{document}

\title[MME for certain skew products]{exponential mixing of measures of maximal entropy for certain skew products}

\author[K. Marin]{Karina Marin}
\address{Departamento de Matem\'atica, Universidade Federal de Minas Gerais (UFMG), Av. Ant\^onio Carlos 6627, 31270-901. Belo Horizonte -- MG, Brasil}
\email{kmarin@mat.ufmg.br}

\author[M. Poletti]{Mauricio Poletti}
\address{Departamento de Matem\'atica, Universidade Federal do Cear\'a (UFC), Campus do Pici,
Bloco 914, CEP 60455-760. Fortaleza -- CE, Brasil}
\email{mpoletti@mat.ufc.br}

\author[F. Veiga]{Filiphe Veiga}
\address{Departamento de Matem\'atica, Universidade Federal de Minas Gerais (UFMG), Av. Ant\^onio Carlos 6627, 31270-901. Belo Horizonte -- MG, Brasil}
\email{filiphx@ufmg.br}

\thanks{
KM was partially supported by FAPEMIG, CAPES-Finance Code 001, Instituto Serrapilheira grant number Serra-R-2211-41879. MP was partially supported by CAPES-Finance Code 001, Instituto Serrapilheira grant number Serra-R-2211-41879 and FUNCAP grant AJC 06/2022. FV was supported by CAPES}

\begin{abstract}
 We establish a relation between the continuity of the fiber entropy and the continuity of the fiber Lyapunov exponents for skew products with 2-dimensional fibers. This result extends the theorem for surfaces proved by Buzzi-Crovisier-Sarig. As a consequence, we are able to obtain classes of skew products that satisfies the strong positive recurrence (SPR) property, in particular these maps have finite number of measures of maximal entropy, all exponentially mixing with good statistical properties. 
\end{abstract}

\maketitle
\section{introduction}
The complexity of a dynamical system can be measured by a number called entropy. There are two types of entropy that are measured on a continuous transformation $T\colon X\to X$ that preserves an invariant probability measure $\mu$, the \emph{topological entropy} $h_{top}(T)$ and the \emph{metric entropy} $h_\mu(T)$. The topological entropy only depends on $T$ and it turns out that, if $X$ is compact, the topological entropy is the supremum of the metric entropy among all the invariant probabilities measures,
$$
h_{top}(T)=\sup\{h_\mu(T):T_*\mu=\mu\}.
$$
A measure that attains this supremum is called a \emph{Measure of Maximal Entropy (MME)}.

The study of measures of maximal entropy is one of the main topics in dynamical systems. The principal questions are:
\emph{Does there exists a measure of maximal entropy?} \emph{If it exists, is it unique?}  \emph{Are there finite many ergodic ones?} \emph{What statistical properties do they satisfy?}

In the case of Anosov diffeomorphisms, this is well understood. There are finite many MMEs and if the map is transitive, it admits a unique one. In both cases, they all have good statistical properties. See \cite{ruelle1978}, \cite{parry1990}, \cite{gouezel2010}, \cite{guivarc1988}, \cite{kifer1990}. Also there are results for some class of non Anosov systems that satisfies some non-uniform form of hyperbolicity, \cite{young1998}, \cite{young1999}. 

Recently Buzzi, Crovisier and Sarig proved that $C^\infty$ transitive diffeomorphisms with positive entropy have a unique MME \cite{BCS-finite} and it satisfies good statistical properties \cite{BCS-SPR}.

In this work, we deal with skew products, that is, maps of the form
$$
F:M\times N\to M\times N,\quad (x,y)\mapsto (f(x),g_x(y)).
$$

We obtain the following result.

\begin{theo}\label{theo-intro-2}
Let $M$ be a manifold, $N$ be a surface, and $F:M\times N\to M\times N$ a skew product diffeomorphism over $f:M\to M$ such that $f$ is a $C^{1+}$ Anosov diffeomorphism and $F$ acts $C^\infty$ on the fibers. If $h_{top}(F)>h_{top}(f)$, then $F$ has a finite number of ergodic MMEs. Moreover, for each of them there exists $p$ such that $f^p$ has exponential decay of correlations, large deviation, and almost sure invariance principle for the Birkhoff sums.
\end{theo}

For skew products with circle fibers, $N=S^1$, the finiteness of MME was proved to be generic when $F$ is partially hyperbolic and $f$ Anosov in \cite{HHTU} and in \cite{MPP} when $f$ is in a class of Derived from Anosov. It is important to notice that in this work we do not require $F$ to be partially hyperbolic.

When $M$ is a surface and $f\in C^\infty$, we can remove the hypothesis of $f$ being Anosov adding a condition on the entropy. Let $h_{top}(F\mid \pi)$ be the supremum of the topological entropy along the fibers $\{x\}\times N$. 

\begin{theo}\label{theo-intro-1}
Let $M$ and $N$ be surfaces and $F:M\times N\to M\times N$ be a $C^\infty$ skew product diffeomorphism over $f:M\to M$. If $h_{top}(F)>\max\{h_{top}(f), h_{top}(F\mid\pi)\}$, then $F$ has a finite number of ergodic MMEs. Moreover, for each of them there exists $p$ such that $f^p$ has exponential decay of correlations, large deviation, and almost sure invariance principle for the Birkhoff sums.
\end{theo}

Both results are consequence of a condition called SPR introduced by Buzzi, Crovisier and Sarig in \cite{BCS-SPR}. 

Let $f:M\to M$ be a $C^{1+}$ diffeomorphism on a compact manifold $M$. We recall the definition of Pesin block.

\begin{defn}
Fix \(\chi, \varepsilon > 0\). A \((\chi, \varepsilon)\)-Pesin block is a non-empty set \(\Lambda \subset M\) for which there is a direct sum decomposition \(T_xM = E^s(x) \oplus E^u(x)\) for all \(x \in \bigcup_{n \in \mathbb{Z}} f^n(\Lambda)\), and a uniform number \(C > 0\) such that for any \(n \in \mathbb{Z}\), \(j \geq 0\), and \(y \in \Lambda\), \[
\max\left(\|Df^j|_{E^s(f^n(y))}\|, \|Df^{-j}|_{E^u(f^n(y))}\|\right) \leq C \exp(-\chi j + \varepsilon |n|).
\]
\end{defn}

The SPR property defined in \cite{BCS-SPR} is the following:
\begin{defn}\label{spr}
A diffeomorphism \(f\) of a closed manifold is \textbf{strongly positively recurrent (SPR)}, if there exists \(\chi > 0\) such that for each \(\varepsilon > 0\), there are a Borel \((\chi, \varepsilon)\)-Pesin block \(\Lambda\) and numbers \(h_0 < h_{\text{top}}(f)\) e \(\tau > 0\) as follows:

For every ergodic measure \(\nu\),  
\[
h_\nu(f) > h_0 \implies \nu(\Lambda) > \tau.
\]
\end{defn}

In \cite{BCS-SPR} they proved that SPR diffeomphisms have a finite number of measures of maximal entropy each of them with good statistical properties, as:

\textbf{Exponential decay of correlation for $f^p$}: There exists $p\geq 1$ such that, for every \( \beta > 0 \) there are \( 0 < \theta < 1 \) and \( C > 1 \) such that for all \( \varphi, \psi \) which are \( \beta \)-H\"older, and for every ergodic component \( \mu' \) of \( (\mu, f^p) \), \[
\left| \int \varphi \cdot (\psi \circ f^{np}) \, d\mu' - \int \varphi \, d\mu' \int \psi \, d\mu' \right| \leq C \|\varphi\|_{\beta}' \|\psi\|_{\beta}'' \theta^{np} \quad (\forall n \geq 0).
\]

\textbf{Large deviation of Birkhoff sums}: For $\beta>0$ there exists $c>0$ such that for every $\beta$-H\"older $\psi$ with H\"older norm $1$ and non zero variance $\sigma_{\psi}$, we have
$$
\lim_{n \to \infty} \frac{1}{n} \log \mu \{ x : \psi_n(x) \geq na \} = -I_\psi(a) \text{ for all } 0 < a < c\sigma_\psi^4.
$$
where $\psi_n(x)=\sum_{j=0}^{n-1}\psi(f^j(x))$ and $I_\psi(a)=\frac{a^2}{2\sigma^2_\psi}(1+o(1))$ when $a\to 0$.

\textbf{Almost sure invariance principle}: The stochastic process  $(S_n)_{n \geq 1}$ satisfies the \emph{almost sure invariance principle} with parameter $\sigma \geq 0$ and rate $o(n^\gamma)$ for $0 < \gamma < \frac{1}{2}$, if there exist two stochastic processes $(\widetilde{S}_n)_{n \geq 1}$ and $(\widetilde{B}_t)_{t \geq 0}$ defined on a common standard probability space such that
\begin{enumerate}
    \item the stochastic processes $(\widetilde{S}_n)_{n \geq 1}$ and $(S_n)_{n \geq 1}$ are equal in distribution;
    \item $(\widetilde{B}_t)_{t \geq 0}$ is a standard Brownian motion;
    \item $|\widetilde{S}_n - \sigma \widetilde{B}_n| = o(n^\gamma)$ a.e. as $n \to \infty$.
\end{enumerate}
This is satisfied for $S_n(x)=\sum_{j=0}^{n-1}\psi(f^j(x))$ where $\psi$ is a H\"older function.

There are other implications like central limit theorem, convergence of momentum, etc. For the full list of stochastic consequences of SPR, see \cite[Theorem~E]{BCS-SPR}.

Theorems~\ref{theo-intro-1} and \ref{theo-intro-2} are consequences of the following results.
\begin{ltheorem}\label{thm-anosov}
Let $M$ be a manifold, $N$ be a surface and $F:M\times N\to M\times N$ a skew product diffeomorphism over $f:M\to M$ such that $f$ is a $C^{1+}$ Anosov diffeomorphism and $F$ acts $C^\infty$ on the fibers. If $h_{top}(F)>h_{top}(f)$, then $F$ is SPR.
\end{ltheorem}

\begin{ltheorem}\label{thm-smooth-2}
Let $M$ and $N$ be surfaces and $F:M\times N\to M\times N$ be a $C^\infty$ skew product diffeomorphism over $f:M\to M$. If $h_{top}(F)>\max\{h_{top}(f),h_{top}(F\mid\pi)\}$, then $F$ is SPR.
\end{ltheorem}

In order to establish these theorems, we prove our main technical result, Theorem~\ref{thm-main}. Its statement relates the continuity of the fiber entropy to the continuity of the fiber Lyapunov exponents for skew products with two-dimensional fibers. 

In the next section, we introduce the definitions needed to precisely state Theorem~\ref{thm-main}. Our result can be viewed as an adaptation of the main theorem of \cite{BCS}, which studies the continuity of Lyapunov exponents with respect to entropy for surface diffeomorphisms. Since \cite{BCS} is a long and highly technical paper, rather than repeating all the calculations, we focus in the key adaptations and differences required in our case. We would like to highlight two of them: the entropy formulas for the fiber entropy developed in Section 5 and the reparametrizations results in Section 6.

\section{Preliminaries and statement of Theorem \ref{thm-main}}\label{sec.preliminaries}
In the following, we let $M$ be a compact metric space and $N$ denote a smooth, compact Riemannian manifold without boundary. Additionally, we assume that $N$ is a surface.

\subsection{Skew products} We denote by $\operatorname{Diff}^\infty(N)$ the space of $C^{\infty}$ diffeomorphisms of $N$, equipped with the Whitney $C^{\infty}$ topology.

Let $f: M \rightarrow M$ be a homeomorphism and $g_x\in \operatorname{Diff}^\infty(N)$ for every $x\in M$.

\begin{defn}[Skew product] By a skew product $F$ on \(M\times N\) we mean a homeomorphism \(F: M\times N \rightarrow M\times N\) defined by
\begin{align*}
F:M\times N\;&\longrightarrow \;M\times N\\
(x, y)\;\; &\mapsto \left(f(x), g_x(y)\right),
\end{align*}
such that the following maps are continuous,
\begin{align*}
\begin{array}{rcl}
M &\longrightarrow& \operatorname{Diff}^{\infty}(N) \\
x &\longmapsto& g_x
\end{array}
\quad & \text{and}\quad
\begin{array}{rcl}
M &\longrightarrow& \operatorname{Diff}^{\infty}(N) \\
x &\longmapsto& g^{-1}_x.
\end{array}
\end{align*}
\end{defn}

\begin{defn}
Let $F_k(x, y) = \left(f_k(x), g_{k,x}(y)\right)$. We say that the sequence of skew-products $F_k$ converges to $F(x, y) = \left(f(x), g_x(y)\right)$ if the following conditions are satisfied:
\begin{enumerate}
    \item $f_k \to f$ uniformly, and
    \item $g_{k,x} \to g_x$ in the $C^\infty$ topology uniformly in $x\in M$.
\end{enumerate}
\end{defn}
\subsection{Lyapunov exponents} Assume that the skew product map $F$ is equipped with an $F$-invariant measure $\mu$. According to Furstenberg-Kesten's Theorem \cite{FK}, the following limits, 
$$
\begin{aligned}
& \lambda_1^c(F,x,y)=\lim_{n \rightarrow \infty} \frac{1}{n} \log \left\|D_y g_x^n \right\|, \\
& \lambda_2^c(F,x,y)=\lim_{n \rightarrow \infty} \frac{1}{n} {\log \left\|D_y (g_x^{n})^{-1}\right\| },^{-1}
\end{aligned}
$$
exists for $\mu$-a.e. $(x,y)$, where
$$
g_x^n= \begin{cases}\mathrm{Id} & \text { if } n=0 \\ g_{f^{n-1}(x)} \circ \cdots \circ g_{f(x)} \circ g_{x} & \text { if } n>0 \\\left(g_{f^n(x)}\right)^{-1} \circ \cdots \circ\left(g_{f^{-1}(x)}\right)^{-1}& \text { if } n<0\end{cases}
$$
and
$$
\lambda_1^c(F,x,y)\geq \lambda_2^c(F,x,y).
$$
The numbers $\lambda_j^c(F,x,y)$ are called \textit{fiber Lyapunov exponents of} $F$ with respect the measure $\mu$. For $j=1,2$, define $$\lambda_j^c(F,\mu)=\int \lambda_j^c(F,x,y)\; d\mu.$$

If $\lambda_1^c(F,x,y)\neq \lambda_2^c(F,x,y)$ then, by Oseledets Theorem \cite{O}, there is a measurable splitting
$$
\{x\} \times T_y N = E_1(x, y) \oplus E_2(x, y),
$$
such that
$$
E_j{}(x,y):=\left\{v \in T_y N\setminus \{0\} : \lim_{n \rightarrow \pm \infty} \frac{1}{n} \log \left\|D_y g_x^n v\right\| =\lambda_j^c(F, x,y)\right\} \cup\{0\}.
$$
The subspaces $E_j(x, y)$ are invariant under $F$. 

If $\mu$ is ergodic, then for $\mu$-a.e. $(x,y)$, $\lambda_j^c(F,x,y)=\lambda_j^c(F, \mu)$ and the dimensions of the subspaces $E_j(x, y)$ are constant. 

\begin{defn}
Let $F$ be a skew product and let $\mu$ be an $F$-invariant and ergodic measure. We say that $\mu$ is of \textit{saddle} type if $\lambda_1^c(F,\mu)>0>\lambda_2^c(F,\mu)$.
\end{defn}

\subsection{Fiber entropy} 

Let $\pi$ denote the canonical projection $\pi: M \times N \rightarrow M$. For each $x \in M$, we define the fiber of the product manifold $M \times N$ over $x$ as
$$
N_x := \{x\} \times N = \pi^{-1}(x).
$$

Let $\mu$ be an $F$-invariant probability measure and $\nu = \pi_*\mu$. Consider the family \(\{\mu_x\}_{x \in M}\) of conditional measures along the fibers \(N_x\) (see Rokhlin \cite{Rokhlin}). This family is uniquely determined up to a \(\nu\)-null set.

Given a measurable partition $\mathcal{P}$ of $N$, define
$$
\mathcal{P}_x^n=\bigvee_{j=1}^n g_x^{-1} \circ g_{f(x)}^{-1} \cdots \circ g_{f^{j-1}(x)}^{-1} \mathcal{P}.
$$
Suppose that
\begin{equation}\label{E1}
\int_{M} H_x(\mathcal{P}) d \nu<\infty
\end{equation}
where $H_x(\mathcal{P})=-\sum_{P \in \mathcal{P}} \mu_x(P) \log \mu_x(P)$.

\begin{theo}[\cite{Abramov1962}] \label{thm-ab-rok}
Let $F$ be a skew product, $\mu$ an $F$-invariant probability measure and $\nu = \pi_*\mu$. For every \(\mathcal{P}\) satisfying (\ref{E1}),
\[
h_\mu(F \mid \pi, \mathcal{P}) := \lim_{n \rightarrow \infty} \frac{1}{n} \int_M H_x\left(\mathcal{P}_x^n\right) d\nu.
\]  
This limit exists and is finite. Moreover, if we define
\[
h_\mu(F \mid \pi) = \sup_\mathcal{P} h_\mu(F \mid \pi, \mathcal{P}),
\]
then,
\begin{equation}\label{E2}
h_\mu(F) = h_{\nu}(f) + h_\mu(F \mid \pi).
\end{equation}
\end{theo}
The quantity \(h_\mu(F \mid \pi)\) is called the \textit{fiber entropy} of the skew product $F$.

\begin{theo}[\cite{Bahnmuller1995}]\label{T2}
For a skew product $F$ and an $F$-invariant and ergodic measure \(\mu\), the following Margulis-Ruelle inequality type holds,
\[
h_\mu(F \mid \pi) \leq \sum_{\lambda_j^c(F,\mu)> 0} \lambda_j^c(F,\mu) \operatorname{dim} E_j.
\]
\end{theo}

Since the  fiber entropy coincides with the conditional entropy of $F$ with respect to $\pi^{-1} \mathcal{B}$, where $\mathcal{B}$ is the Borel $\sigma$-algebra of $f$, (see Kifer \cite{Kifer1986}), and the \(\sigma\)-algebra generated by \(\pi^{-1}\) is \(F\)-invariant, then it follows that the fiber entropy satisfies the equality
\begin{equation}
h_\mu(F \mid \pi) = h_\mu\left(F^{-1} \mid \pi\right).
\end{equation}

In particular, we obtain the reverse Margulis-Ruelle inequality,
\[
h_\mu(F \mid \pi) \leq \sum_{\lambda_j^c(F,\mu)< 0} -\lambda_j^c(F,\mu) \operatorname{dim} E_j.
\]
\subsection{Main results}
\begin{ltheorem}\label{thm-main}
Let $M$ be a compact metric space and $N$ a closed smooth surface. Let $F_k$, $k\in \mathbb{N}$, be a skew product and let $\mu_k$ be an $F_k$-invariant and ergodic measure. Suppose that:
\begin{enumerate}[label = -]
    \item the limits $\lim_k \lambda_1^c\left(F_k, \mu_k\right)$ and $\lim _k h_{\mu_k}\left(F_k\mid \pi\right)$ exist and are positive,
    \item $F_k$ converges to a skew product $F$,
    \item ${\mu}_k \xrightarrow{w^*} {\mu}$ for some ${F}$-invariant measure ${\mu}$.
\end{enumerate}
Then, there exist $\beta \in(0,1]$ and two $F$-invariant measures $\mu_0$ and $\mu_1$ such that $\mu=(1-\beta)\mu_0+\beta \mu_1$, $h_{\mu_1} (F\mid \pi) > 0$ and,
$$
\lim _{k \rightarrow \infty} h_{\mu_k}\left(F_k\mid\pi\right)\leq \beta h_{\mu_1}\left(F\mid \pi\right).
$$
Moreover, if $\lambda_2(F,x,y)\leq0$ for $\mu$-a.e. $(x,y)$.
$$
\lim _{k \rightarrow \infty} \lambda_1^c(F_k,\mu_k)=\beta \lambda_1^c\left(F, \mu_1\right).
$$
\end{ltheorem}

\begin{coro}
    Let $F_k$ be a skew product and let $\mu_k$ be an $F_k$-invariant and ergodic measure. Suppose that $F_k$ converges to a skew product $F$ and ${\mu}_k \to{\mu}$ in the weak* topology for some ${F}$-invariant and ergodic probability measure ${\mu}$.
    If
    $$
    h_{\mu_k}\left(F_k\mid\pi\right)\to h_{\mu}\left(F\mid \pi\right)>0
    $$
    then
    $$
     \lambda_1^c(F_k,\mu_k) \to  \lambda_1^c\left(F, \mu\right)\quad \text{and}\quad \lambda_2^c(F_k,\mu_k) \to  \lambda_2^c\left(F, \mu\right).
    $$

\end{coro}
\section{The projective fiber bundle}
\label{sec: PROJECTIVE FIBER BUNDLE}
Consider $\widehat{N}$ to be the  fiber bundle $(\widehat{N}, \widehat{\pi}, N)$ where $\widehat{\pi}_{\widehat{N}}: \widehat{N} \rightarrow N$ is the natural projection $\widehat{\pi}_{\widehat{N}}(x, E)=x$, and 
$$\widehat{N}:=\left\{(x, E): x \in N, E\right.\;\text{is a one-dimensional linear subspace of}\;T_xN\}$$

Since \(\dim N = 2\), the manifold \(\widehat{N}\) is a smooth, compact three-dimensional manifold. We endow it with the Riemannian metric given by \(\sqrt{ds^2 + d\theta^2}\), where \(ds\) is the length element on \(N\) and \(d\theta\) is the length element on $\widehat{N}_x$. We call
$M\times \widehat{N}$ the \textit{projective fiber bundle of} $M\times N$.

\begin{defn}[Lift]
Given a skew product $F$ on $M\times N$, the \textit{canonical lift of} $F$ to $M\times \widehat{N}$ is defined as 
$$
\widehat{F} (x,(y,E)) = \bigl(f(x),\widehat{g}_x(y,E)\bigr)
$$  
where 
$$
\widehat{g}_x(y,E) = \bigl(g_x(y), D_yg_x (E)\bigr).
$$
\end{defn}
Every $\widehat{F}$-invariant probability measure $\widehat{\mu}$ on $M\times \widehat{N}$ projects to an $F$-invariant probability measure $\mu$ on $M\times N$ given by
$$\mu:=\bigl(\operatorname{Id},\widehat{\pi}_{\widehat{N}}\bigr)_*(\widehat{\mu}):=\widehat{\mu} \circ \bigl(\operatorname{Id},\widehat{\pi}_{\widehat{N}}\bigr)^{-1}
$$ where \begin{gather*} \bigl(\operatorname{Id},\widehat{\pi}_{\widehat{N}}\bigr): M\times \widehat{N}\to {M\times N}\\
(x,(y,E))\mapsto (x,y).
\end{gather*}
We call $\mu$ the $\textit{projection}$ of $\widehat{\mu}$. Analogously, given $\mu$ an $F$-invariant measure, we say that $\widehat{\mu}$ is a $\textit{lift}$ of $\mu$ if $\widehat{\mu}$ is $\widehat{F}$-invariant and projects to $\mu$.




It is not difficult to see that if the fiber Lyapunov exponents of an ergodic measure $\mu$ are distinct, then there are precisely two ergodic \(\widehat{F}\)-invariant lifts of $\mu$,
\begin{equation}\label{E4}
\widehat{\mu}^{+}:=\int_{M\times N} \delta_{\left(x,(y , E_1(x,y))\right)} d \mu \quad \text { and } \quad\widehat{\mu}^{-}:=\int_{M\times N} \delta_{\left(x, (y,(E_2(x,y))\right)} d \mu \text {. }
\end{equation}

Moreover, if $\widehat{\varphi}(x,(y,E)) =\log \|D_yg_x|_E\|$, then 
\begin{equation}\label{E3}
\int \widehat{\varphi}\, d \widehat{\mu}^{+}=\lambda_1^c(F, \mu)\quad \text{and}\quad\int \widehat{\varphi}\, d \widehat{\mu}^{-}=\lambda_2^c(F, \mu).
\end{equation}

\section{The discontinuity ratio of the fiber exponents}
Following the definitions given in \cite{BCS}, we introduce the notions of empirical measures and neutral blocks.
\begin{defn}[Empirical measures]
Let $T$ be a homeomorphism on a compact metric space $X$. Given $\mathfrak{N} \subset \mathbb{N}$, let
$$
\mu_{x, n}^{\mathfrak{N}}:=\frac{1}{n} \sum_{j \in[0, n) \cap \mathfrak{N}} \delta_{T^j(x)}
$$

The weak-* limit points of $\left(\mu_{x, n}^{\mathfrak{N}}\right)_{n \geq 1}$ are called the $\mathfrak{N}$-empirical measures of $x$.  
\end{defn}

\begin{defn}[Neutral Blocks]
Let \(\psi: X \rightarrow \mathbb{R}\) be a continuous function and let \(\alpha > 0\) and \(L \geq 1\). An interval of integers \((n_0, n_0 + 1, \ldots, n_1 - 1)\) is called an \((\alpha, L)\)-neutral block of \((x, T, \psi)\) if it satisfies the following conditions:
\begin{enumerate}
    \item  $n_1-n_0 \geq L$, and
    \item $\psi\left(T^{n_0}(x)\right)+\psi\left(T^{n_0+1}(x)\right)+\cdots+\psi\left(T^{n-1}(x)\right) \leq \alpha \cdot\left(n-n_0\right)$ for all $n_0 < n \leq n_1$.
\end{enumerate}
\end{defn}
Denote \(\mathfrak{N}_{\alpha, L}(x, T, \psi)\) the collection of all \((\alpha, L)\)-neutral blocks of \((x, T, \psi)\).

In \cite{BCS}, the authors characterized the discontinuity ratio of the Lyapunov exponents as the proportion of time that typical orbits spend outside neutral blocks. More precisely, they proved that,
\begin{prop}[\textcolor{red}{Proposition 6.2 in} \cite{BCS}]\label{P1}
Let \(T, T_1, T_2, \ldots\) be homeomorphisms on a compact metric space \(X\), and let \(\psi, \psi_1, \psi_2, \ldots\) be continuous functions on \(X\) such that 
$$
T_k \to T \quad\text{and}\quad \psi_k \to \psi\quad \text{uniformly}.
$$
For each \(k\), let \(\mu_k\) be an ergodic probability measure for \(T_k\) with the property that \(\int \psi_k \, d\mu_k \geq 0\). Then, there exists a subsequence \((\mu_{k_i})\) and positive measures \(m_0\) and \(m_1\) such that:
\begin{enumerate}[label = (\roman*)]
    \item Both \(m_0\) and \(m_1\) are invariant under \(T\).
    \item The subsequence \((\mu_{k_i})\) converges weak-* to \(m_0 + m_1\).
    \item Let $V = (V_0, V_1)$ be neighborhoods of the measures $m_0$ and $m_1$, respectively. Then there exist constants $\alpha_*(V) > 0$ and $L_*(V) \in \mathbb{N}$ such that for every $\alpha \in (0, \alpha_*(V))$ and every $L \geq L_*(V)$, there exists $i_*(V, \alpha, L) \in \mathbb{N}$ such that the following holds:

For every $i \geq i_*(V, \alpha, L)$, we have that for $\mu_{k_i}$-almost every $x \in X$, the $\mathfrak{N}_{\alpha, L}(x, T_{k_i}, \psi_{k_i})$-empirical measures lie in $V_0$, and the $\mathbb{N} \setminus \mathfrak{N}_{\alpha, L}(x, T_{k_i}, \psi_{k_i})$-empirical measures lie in $V_1$.
    \item We have \(\int \psi \, dm_0 = 0\).
    \item For \(m_1\)-almost every point \(x\), the limit \(\lim_{n \to \infty} \frac{1}{n} \sum_{j=0}^{n-1} \psi(T^j(x))\) is positive.\label{item:P2}
\end{enumerate}
\end{prop}

\begin{prop}\label{expo}
Let $F_k$ be a skew product and $\mu_k$ $F_{k}$-invariant and ergodic measures of saddle type. Suppose that
\begin{enumerate}[label = \emph{(\roman*)}]
\item $\lim_{k\to \infty} \lambda_1^c\left(F_k, \mu_k\right)>0$,
\item $F_k$ converges to a skew product $F$,
\item $\mu_k \xrightarrow{w^*} \mu$ for some $F$-invariant measure $\mu$.
\end{enumerate}
Then, there exists $\beta \in(0,1]$ and two $F$-invariant measures $\mu_0$ and $\mu_1$ such that $\mu=(1-\beta)\mu_0+\beta \mu_1$ and some lift $\widehat{\mu}_1$ of $\mu_1$ such that
$$\lim _{k \rightarrow \infty} \lambda_1^c\left(F_k, \mu_k\right)=\beta\, \int\widehat{\varphi} d\widehat{\mu}_1.$$
Moreover if $\lambda_2(F,x,y)\leq 0$ for $\mu_1$-a.e. $(x,y)$ then
\begin{gather*}
\lim _{k \rightarrow \infty} \lambda_1^c\left(F_k, \mu_k\right)=\beta\, \lambda_1^c(F, \mu_1).
\end{gather*}
\end{prop}
\begin{proof}
Recall that $\widehat{\mu}^+_k$ denotes the lift of $\mu_k$, as defined in Equation (\ref{E4}). Since $M\times\widehat{N}$ is compact, up to taking a subsequence, we can assume that $\widehat{\mu}_k^{+} \xrightarrow{w*} \widehat{\mu}$ for some measure $\widehat{\mu}$. 
Apply Proposition \ref{P1} to $\widehat{F}_k, \widehat{F}, \widehat{\mu}_k^+$ and $\widehat{\varphi}_k, \widehat{\varphi}: M\times\widehat{N}\to \mathbb{R}$ defined by
\begin{gather*}
\widehat{\varphi}_k(x,(y,E)) =\log \|D_yg_{k,x}|_E\|\quad \\ \text{and} \\
\widehat{\varphi}(x,(y,E)) =\log \|D_yg_x|_E\|
\end{gather*}
where $x\mapsto g_{k,x}$ is the action on the fiber related to the skew product $F_k$. By Equation (\ref{E3}), we know that
$$
\int \widehat{\varphi}_k \,d \widehat{\mu}_k^{+}=\lambda_1^c\left(F_k, \mu_k\right)>0.
$$

Let $\widehat{m}_0$ and $\widehat{m}_1$ be the measures given by Proposition \ref{P1}. Then we have $\widehat{\mu} = \widehat{m}_0 + \widehat{m}_1$. Define $\beta=1-\widehat{m}_0(\widehat{M})=\widehat{m}_1(\widehat{M})$ and
$\widehat{\mu}_i:=\frac{1}{\widehat{m}_i(\widehat{M})} \widehat{m}_i$, or any invariant probability measure if $\widehat{m}_i(\widehat{M})=0$.

By hypothesis, $\beta\neq 0$. Otherwise
$$
\begin{aligned}
0=\int \widehat{\varphi} \,d \widehat{m}_0 & =\int \widehat{\varphi} \,d \widehat{\mu} \\
& =\lim _{k \rightarrow \infty} \lambda_1^c\left(F_k, \mu_k\right)>0.
\end{aligned}
$$

Define $\mu_i=\left(\operatorname{Id}, \pi_{\widehat{N}}\right)_*\widehat{\mu}_i$. Since $\widehat{\mu} =  (1 - \beta) \widehat{\mu}_0 + \beta \widehat{\mu}_1$ and $\left(\operatorname{Id}, \pi_{\widehat{N}}\right)_*\widehat{\mu} = \mu$, it follows that
$\mu = (1 - \beta) \mu_0 + \beta \mu_1 .$

Therefore we conclude the first part. Now let us assume that $\lambda_2(F,x,y)\leq 0$ $\mu_1$-a.e. $(x,y)$.
We claim that $\widehat{\mu}_1=\widehat{\mu_1}^{+}$ (the unstable lift of $\mu_1$). By Proposition \ref{P1} \ref{item:P2} and the definition of $\widehat{\varphi}$, we conclude that $\lambda^c_1(F,x,y)>0$ for $\mu_1$-almost every point. 


By Oseledets theorem, for $\mu_1$-almost every $(x,y)\in M\times N$,
$$
\{x\}\times T_y N=E_1(x,y)\oplus E_2(x,y)
$$
where $D g_x E_j(x,y)=E_j(f(x),g_x(y))(j=1,2)$.
The measure $\mu_1$ has a unique lift $\widehat{\mu_1}^{+}$ to $\operatorname{graph}\left({E}_1\right)$. Since
$$
\lim _{n \rightarrow \infty} \frac{1}{n} \sum_{j=0}^{n-1} \widehat{\varphi}(\widehat{F}^j(\widehat{x}))\leq 0
$$
on $\operatorname{graph}\left(E_2\right)$ and any other lift of $\mu_1$ charge some part to $E_2$ then, by Proposition \ref{P1} \ref{item:P2}, we have that $\widehat{\mu}_1=\widehat{\mu}^{+}$.

Therefore
$$
\begin{aligned}
\beta \,\lambda^{c}_1\left(F, \mu_1\right) & =\beta \int \widehat{\varphi} \,d \widehat{\mu_1}^{+}  \\
& =\beta \int \widehat{\varphi} \,d \widehat{\mu}_1=\int \widehat{\varphi} \,d \widehat{m}_1 \\
& =\int \widehat{\varphi}\, d \widehat{m}_1+\int \widehat{\varphi}\, d \widehat{m}_0=\int \widehat{\varphi} \,d \widehat{\mu} \\
& =\lim_{k \rightarrow \infty} \lambda_1^c\left(F_k, \mu_k\right).
\end{aligned}
$$

Thus 
$$
\lim _{k \rightarrow \infty} \lambda_1^c\left(F_k, \mu_k\right)=\beta\, \lambda_1^c\left(F, \mu_1\right)
$$
as required.
\end{proof}

\begin{rem}
    In contrast with the main Theorem of \cite{BCS} here in order to conclude the last part of the theorem, we need the hypothesis of $\lambda_2\leq0$, $\mu_1$-a.e. This was a consequence in their result. The main difference is that in the surface diffeomorphism case $\lambda_2>0$ implies the existence of periodic repellers and this is not the case for skew products.
\end{rem}

\section{Entropy formulas}
\label{sec: ENTROPY FORMULAS}

In this section, we present two different formulas for the fiber entropy of a skew product. The main result in Section 5.1 holds for ergodic saddle type measures and it is going to be applied to the sequence of measures $\mu_k$ in Theorem \ref{thm-main}. The formula in Section 5.2 is valid for any invariant measure and it will also be used in the proof of Theorem \ref{thm-main} but now for the limit measure $\mu$.

\subsection{Ledrappier-Young type formulas} 
Let $F$ be a skew product and let $\mu$ be an $F$-invariant ergodic measure of saddle type. Let \(\mathcal{R}_{\mu}\) denote the set of points where the Lyapunov exponents of $\mu$ are well-defined.
\begin{defn}[Fiber-wise unstable manifolds]
The set 
 $$\widetilde{W}^u(x, y):=\left\{z \in N: \limsup_{n \rightarrow+\infty}\frac{1}{n} \log \left.d\left(g_x^{-n} (y), g_x^{-n} (z)\right) \leq\right.-\lambda_1^c(F,\mu)\right\}$$ is called the \textit{fiber-wise unstable manifold} of $F$ at $(x, y)$, where $(x, y)\in \mathcal{R}_{\mu}$.
\end{defn}

Each \( \widetilde{W}^u(x, y) \) is a \( C^{\infty} \) immersed submanifold of \( N \) and $\{x\}\times \widetilde{W}^u(x, y)$ is tangent at \( (x, y) \) to $E_1(x, y)$. 

In this section, we will define quantities referred to as local entropy along fiber-wise unstable manifolds. In the ergodic case, the concept of local entropy along the $\widetilde{W}^u$-manifolds is represented by a number $h^u$, which measures the degree of randomness along the leaves of the $\widetilde{W}^u$-manifolds. This notion of entropy originates from Ledrappier and Young \cite{LY2}. Here, we present a version for skew products as described in \cite{LiuXie2006}.
\begin{defn}[Subordinated partitions]
 A measurable partition $\eta$ of $M\times N$ is said to be subordinate to the $\widetilde{W}^u$-manifolds if for $\mu$-a.e. $(x, y)$, $\eta_x(y):=\{z\in N :(x, z) \in \eta(x, y)\} \subset \widetilde{W}^u(x, y)$ and $\eta_x(y)$ contains an open neighborhood of $y$ in $\widetilde{W}^u(x, y)$.   
\end{defn}

Let \(\eta\) be a partition of \(M\times N\) that is subordinate to the \(\widetilde{W}^u\)-manifolds.  For each \(x \in M\), we define \(\eta_x\) as the partition \(\{\eta_x(y) : y \in N\}\) of \(N\). Recall that $\{\mu_x\}_{x\in M}$ is the disintegration of $\mu$ along the fibers $\{x\}\times N$.  For each $x\in M$, consider 
$$\{(\mu_x)_y^{\eta_x}\}_{y \in N},$$ 
as the canonical family of conditional measures of \(\mu_x\) associated with \(\eta_x\). 

For \(\varepsilon > 0\), \((x, y) \in M\times N\) and \(n \in \mathbb{Z}^{+}\), define  
\[
B^F_x(y ; n, \varepsilon) := \left\{ z \in N : d\left(g_x^j(y), g_x^j(z) \right) < \varepsilon \text{ for all } 0 \leq j < n \right\}.
\]
In a manner similar to the approach taken by Ledrappier and Young \cite{LY2}, the following quantities are introduced:
$$
\begin{aligned}
& \underline{h}^u(x, y ; \varepsilon, \eta):=\liminf _{n \rightarrow+\infty}-\frac{1}{n} \log (\mu_x)_y^{\eta_x}\left(B^F_x(y ; n, \varepsilon)\right), \\
& \overline{h}^u(x, y ; \varepsilon, \eta):=\limsup _{n \rightarrow+\infty}-\frac{1}{n} \log (\mu_x)_y^{\eta_x}\left(B^F_x(y ; n, \varepsilon)\right).
\end{aligned}
$$
These functions are clearly measurable. Additionally, it is important to observe that the functions increase as \(\varepsilon \to 0\). 

Finally, the concepts of lower and upper local entropy along the \(\widetilde{W}^u\)-manifolds at \((x, y)\), with respect to \(\eta\), can be defined. We remark that in \cite{LiuXie2006}, the authors use the intrinsic metric on the unstable sub-manifold to define these quantities. We remark that although we use the global metric, both definitions coincide because $\varepsilon\to 0$.
$$
\begin{aligned}
& \underline{h}^u(x, y ; \eta):=\lim_{\varepsilon \rightarrow 0} \underline{h}^u(x, y ; \varepsilon, \eta) = \lim_{\varepsilon \rightarrow 0} \liminf _{n \rightarrow+\infty}-\frac{1}{n} \log (\mu_x)_y^{\eta_x}\left(B^F_x(y ; n, \varepsilon)\right), \\
& \overline{h}^u(x, y ; \eta):=\lim_{\varepsilon \rightarrow 0} \overline{h}^u(x, y ; \varepsilon, \eta)=\lim_{\varepsilon \rightarrow 0}\limsup _{n \rightarrow+\infty}-\frac{1}{n} \log (\mu_x)_y^{\eta_x}\left(B^F_x(y ; n, \varepsilon)\right).
\end{aligned}
$$

We denote by \(\sigma\) the partition \(\pi^{-1}(\mathcal{E})\), where \(\mathcal{E}\) represents the partition of \(M\) into points. Let $\eta$ be a partition subordinate to the \(\widetilde{W}^u\)-manifolds such that $\eta >\sigma$, \(F^{-1} \eta > \eta\), $F^{-n}\eta$ tends to the partition into single points and $h_\mu(F\mid \pi)=H_{\mu}\left(\eta \mid F(\eta) \vee \sigma\right)$. Such a partition always exists (see Proposition 3.7 in \cite{BahnmullerLiu1998}).

\begin{prop}[Proposition 8.6, \cite{LiuXie2006}]\label{P2}
Let \(\eta\) be a partition satisfying the properties above, then for $\mu$-a.e. $(x, y)$, 
$$
\underline{h}^u\left(x, y ; \eta\right)=\overline{h}^u\left(x, y ; \eta\right)=H_{\mu}\left(\eta \mid F(\eta) \vee \sigma\right).
$$
\end{prop}

Using this result, we conclude that the fiber entropy of a skew product $F$ can be expressed for $\mu$-a.e. $(x,y)$ as the local entropy along the $\widetilde{W}^{u}$-manifolds, that is,
\begin{equation}
    h_\mu(F\mid \pi)=\lim_{\varepsilon \rightarrow 0}\liminf _{n \rightarrow+\infty}-\frac{1}{n} \log (\mu_x)_y^{\eta_x}\left(B^{F}_x( y ; n, \varepsilon)\right).
\end{equation}

Observe that in particular the proposition implies that this limit does not depend on the partition $\eta$.

\begin{defn}
We define the following quantity:
$$r_F(n, \varepsilon,(\mu_x)_{y}^{\eta_x}, \gamma):=\inf\left\{r^F_x(n, \varepsilon, Z):\; Z \subset N \;\text{and}\;(\mu_x)_{y}^{\eta_x}(Z) >\gamma\right\},
$$

where 
$$
r^F_x(n, \varepsilon, Z) = \inf \left\{\#A: \;Z \subset \bigcup_{y\in A} B^F_x(y ; n, \varepsilon)\right\}.
$$
\end{defn}
\begin{prop}\label{P3}
Let $F$ be a skew product and suppose that the measure $\mu$ is an $F$-invariant ergodic measure of saddle type. Then, for $\mu$-almost every point $(x, y) \in M \times N$, we have
\begin{equation}
h_\mu(F\mid \pi)=\inf_{\gamma>0}\lim_{\varepsilon \rightarrow 0}\liminf _{n \rightarrow+\infty}\frac{1}{n} \log r_F\left(n, \varepsilon, (\mu_x)_y^{\eta_x}, \gamma\right).
\end{equation}
\end{prop}
The proof is an adaptation of the arguments in Y. Zang \cite{Zang}. 
\begin{proof}
For brevity, we denote the fiber entropy $h_{\mu}(F \mid \pi)$ simply by $h_F$ throughout this proof. 

Let $\Gamma_\mu$ denote the full measure set of points satisfying Proposition \ref{P2}. Then, we can assume that, for every $(x,y) \in \Gamma_\mu$, the local unstable entropy satisfies:
\begin{align*}
&\qquad\lim _{\varepsilon \rightarrow 0} \liminf_{n \rightarrow+\infty}-\frac{1}{n} \log (\mu_x)_y^{\eta_x}\left(B^F_x(y;n,\varepsilon)\right)=h_{F}.
\end{align*}
Let us fix an arbitrary $\tau > 0$. We define a family of sets $\Lambda_\tau^{\varepsilon}$ for each $\varepsilon > 0$ as follows:
$$
\Lambda_\tau^{\varepsilon} :=\left\{(x,y) \in \Gamma_\mu: \liminf _{n \rightarrow+\infty}-\frac{1}{n} \log (\mu_x)_y^{\eta_x}\left(B^F_x(y; n, 2 \varepsilon)\right)>h_F-\tau\right\}.
$$
It follows that the union of these sets covers $\Gamma_\mu$, that is, $\cup_{\varepsilon>0} \Lambda_\tau^{\varepsilon}\supset\Gamma_\mu$.

Now, fix $\varepsilon>0$ and a point $p=(x_p, y_p) \in \Lambda_\tau^{\varepsilon}$. Define $$\Lambda_{\tau}^{\varepsilon}(x_p):=\left\{y\in N : (x_p,y) \in \Lambda_\tau^{\varepsilon}\right\}.$$ For any $j \in \mathbb{N}$, we further refine this set by defining:
$$
\Lambda_{\tau}^{\varepsilon}(p,j)=\left\{y\in \Lambda_\tau^{\varepsilon}(x_p): (\mu_{x_p})_{y_p}^{\eta_{x_p}}\left(B^F_{x_p}(y; n, 2 \varepsilon)\right) \leq e^{-n\left(h_F-\tau\right)}, \forall n \geq j\right\}.
$$
By construction, these sets are nested, $\Lambda_\tau^{\varepsilon}(p, j) \subset \Lambda_\tau^{\varepsilon}(p, j+1)$, and their union covers $\Lambda_\tau^{\varepsilon}(x_p)$ in terms of the measure $(\mu_{x_p})_{y_p}^{\eta_{x_p}}$:
$$
(\mu_{x_p})_{y_p}^{\eta_{x_p}}\left(\cup_j \Lambda_\tau^{\varepsilon}(p, j)\right)=(\mu_{x_p})_{y_p}^{\eta_{x_p}}\left(\Lambda_\tau^{\varepsilon}(x_p)\right).
$$

Let $\gamma > 0$ and $p=(x_p, y_p)\in \Gamma_\mu$ be arbitrary. Consider $$\Gamma_{\mu}(x_p):=\{y\in N: (x_p,y)\in \Gamma_{\mu}\}.$$ Since the union $\cup_j \Lambda_\tau^{\varepsilon}(p, j)$ covers $\Lambda_\tau^{\varepsilon}(x_p)$ in measure, and $\cup_{\varepsilon>0} \Lambda_\tau^{\varepsilon}(x_p)\supset\Gamma_\mu(x_p)$, we can choose $\varepsilon > 0$ sufficiently small and $L \in \mathbb{N}$ sufficiently large such that:
$$
(\mu_{x_p})_{y_p}^{\eta_{x_p}}\left(\Lambda_\tau^{\varepsilon}(p, j)\right) \geq 1-\frac{\gamma}{2}, \quad \forall j \geq L.
$$

For a given $n \in \mathbb{N}$, let $C_n \subset \Gamma_\mu(x_p)$ be a set such that $|C_n|$ represents the minimum number of $(n, \varepsilon)$-Bowen balls required to cover a subset of $\Gamma_\mu(x_p)$ with measure at least $\gamma$ under the measure $(\mu_{x_p})_{y_p}^{\eta_{x_p}}$. Moreover, assume that $U_n = \bigcup_{\bar{y} \in C_n} B^F_{x_p}(\bar{y}; n, \varepsilon)$ satisfies $(\mu_{x_p})_{y_p}^{\eta_{x_p}}(U_n) \geq \gamma$.

Therefore, we conclude
$$
(\mu_{x_p})_{y_p}^{\eta_{x_p}}\left(\Lambda_\tau^{\varepsilon}(p, n) \cap U_n\right) \geq \frac{\gamma}{2}, \quad \forall n \geq L.
$$
Let $A_n$ be the points $\bar{y} \in C_n$ for which the corresponding ball $B^F_{x_p}(\bar{y}; n, \varepsilon)$ has a non-empty intersection with $\Lambda_\tau^{\varepsilon}(p, n)$. That is, $A_n = \{ \bar{y} \in C_n : B^F_{x_p}(\bar{y}; n, \varepsilon) \cap \Lambda_\tau^{\varepsilon}(p, n) \neq \emptyset\}$.

The union of balls centered in $A_n$ covers the intersection $\Lambda_\tau^{\varepsilon}(p, n) \cap U_n$. For each $\bar{y} \in A_n$, choose an arbitrary point $\widetilde{y} \in B^F_{x_p}(\bar{y}; n, \varepsilon) \cap \Lambda_\tau^{\varepsilon}(p, n)$. Then,
$$
B^F_{x_p}(\bar{y}; n, \varepsilon) \subset B^F_{x_p}(\widetilde{y}; n, 2 \varepsilon).
$$

Hence for $n \geq L$,
\begin{align*}
\frac{\gamma}{2} &\leq (\mu_{x_p})_{y_p}^{\eta_{x_p}}\left(\bigcup_{\bar{y} \in A_n} B^F_{x_p}(\widetilde{y}; n, 2 \varepsilon)\right) \\ 
&\leq r_F\left(n, \varepsilon,(\mu_{x_p})_{y_p}^{\eta_{x_p}},\gamma\right) \times \sup _{y \in \Lambda_\tau^{\varepsilon}(p, n)} (\mu_{x_p})_{y_p}^{\eta_{x_p}}\left(B^F_{x_p}(y; n, 2 \varepsilon)\right) \\
&\leq r_F\left(n, \varepsilon,(\mu_{x_p})_{y_p}^{\eta_{x_p}},\gamma\right) \times e^{-n\left(h_F-\tau\right)}.
\end{align*}

Therefore for any $\varepsilon= \varepsilon(\gamma)$ small enough and $n$ large enough,
$$
r_F\left(n, \varepsilon,(\mu_{x_p})_{y_p}^{\eta_{x_p}},\gamma\right) \geq \frac{\gamma}{2} \cdot e^{n\left(h_F-\tau\right)}.
$$
Since $\gamma>0$ and $\tau>0$ were arbitrary, it follows that
$$
\begin{aligned}
&h_{F} \leq \inf _\gamma \lim _{\varepsilon \rightarrow 0} \liminf _{n \rightarrow+\infty} \frac{1}{n} \log r_F(n, \varepsilon,(\mu_{x_p})_{y_p}^{\eta_{x_p}},\gamma).
\end{aligned}
$$
The proof of the other inequality proceeds analogously using the identity with $\limsup$.
\end{proof}

\subsection{Entropy formulas on the projective bundle}
In order to analyze entropy in a way that is more compatible with the projective bundle structure, \cite{BCS} introduce a two-scale framework. First, let us recall some definitions.

Let $F$ be a skew product and recall Definition 3.1 of the lift $\widehat{F}$ acting on $M\times \widehat{N}$. That is, $$
\widehat{F} (x,(y,E)) = \bigl(f(x),\widehat{g}_x(y,E)\bigr),
$$  
where 
$$
\widehat{g}_x(y,E) = \bigl(g_x(y), D_yg_x (E)\bigr).
$$

Fix $x\in M$ and a point $\widehat{y} = (y, E) \in \widehat{N}$. Given two positive scales $\varepsilon, \widehat{\varepsilon} > 0$, for each $n \in \mathbb{N}$, we define the $x$-fibered $(n, \varepsilon, \widehat{\varepsilon})$-Bowen ball centered at $\widehat{y} \in \widehat{N}$ as the set
$$
\begin{aligned}
B^{\widehat{F}}_x(\widehat{y}, n, \varepsilon, \widehat{\varepsilon}) := \Big\{ \widehat{z} \in \widehat{N} : &\; d\bigl(\widehat{g}^j_x(\widehat{z}), \widehat{g}^j_x(\widehat{y})\big) < \widehat{\varepsilon} \text{ and } \\
&\; d(g^j_x(z)), g^j_x(y)) < \varepsilon \text{ for all } 0 \leq j < n \Big\}.
\end{aligned}
$$

Given a subset $\widehat{Z} \subset \widehat{N}$, we denote by
$$r^{\widehat{F}}_x(n, \varepsilon, \widehat{\varepsilon}, \widehat{Z})
$$ the $(n, \varepsilon, \widehat{\varepsilon})$-covering number, which is the minimal number of $x$-fibered $(n, \varepsilon, \widehat{\varepsilon})$-Bowen balls of order $n$ needed to cover $\widehat{Z}$.

Let $\widehat{\mu}$ be an $\widehat{F}$-invariant measure and fix $0 < \gamma < 1$. Consider the family $\{\widehat{\mu}_x\}_{x \in M}$ of conditional measures defined on the fibers $\{x\} \times \widehat{N}$. We then define the $(n, \varepsilon, \widehat{\varepsilon}, \gamma)$-covering number of $\widehat{\mu}_x$ as the minimal number of $x$-fibered $(n, \varepsilon, \widehat{\varepsilon})$-Bowen balls needed to cover a set of $\widehat{\mu}_x$-measure at least $\gamma$. This quantity is denoted by $$r^{\widehat{F}}_x(n, \varepsilon, \widehat{\varepsilon}, \widehat{\mu}_x, \gamma).
$$

For an $F$-invariant measure $\mu$, we can define, analogously, $r^{F}_x(n, \varepsilon, \mu_x, \gamma).$ If $\widehat{\mu}$ is a lift of $\mu$, Claim 4.4 in \cite{BCS} extends to the present context to obtain:

\begin{claim}\label{twoscales} Fix $x\in M$. For any $\widehat{\varepsilon},\alpha>0$, there exist $C, \varepsilon_{*}>0$ such that, for any $0<\varepsilon\leq\varepsilon_{*}$ and any $\gamma>0$, $$r^{\widehat{F}}_x(n, \varepsilon, \widehat{\varepsilon}, \widehat{\mu}_x, \gamma)\leq C e^{\alpha n}r^{F}_x(n, \varepsilon, \mu_x, \gamma).$$
\end{claim}

Given an $F$-invariant measure $\mu$, we define the essential fiber entropy by $$\bar{h}_\mu(F\mid \pi):=\text{ess-sup}\; h_{\mu_e}(F\mid \pi),$$ where the supremum is taken over the ergodic components of $\mu$.

\begin{prop}\label{katok} Let $F$ be a skew product and $\mu$ an $F$-invariant measure such that $\pi_*\mu=\nu$. Given $\varepsilon>0$ and $\gamma>0$, for $\nu$-a.e $x\in M$, $$\limsup_{n\to \infty} \frac{1}{n} \log r^{F}_x(n, \varepsilon, \mu_x, \gamma)\leq \bar{h}_\mu(F\mid \pi).$$
\end{prop}
\begin{proof}
    Let $\mathcal{P}$ be a partition of $N$. For $x\in M$ and $n\in \mathbb{N}$, consider the partition $\mathcal{P}^n_x$ as in Theorem \ref{thm-ab-rok}. Given $y\in N$, denote $\mathcal{P}^n_x(y)$ the element of $\mathcal{P}^n_x$ that contains $y$. 

    When the measure $\mu$ is ergodic, Belinskaja \cite{B} established a fibered version of the Shannon-McMillan-Breiman theorem which states that for $\mu$-a.e. $(x,y)$, $$\lim_{n\to \infty}-\frac{1}{n} \log \mu_{x}(\mathcal{P}^n_x(y))=h_\mu(F \mid \pi, \mathcal{P}).$$ Therefore, in this case the proposition follows analogously to the proof of Katok entropy formula \cite{Katok1980}.

    For the general case, it is possible to obtain a non-ergodic version of Belinskaja's theorem using the same reasoning than in \cite{DGS} to prove the non-ergodic version of Shannon-McMillan-Breiman. Then, the conclusion follows again by Katok's arguments.  
\end{proof}

Proposition \ref{katok} and Claim \ref{twoscales} implies the following result:

\begin{coro}\label{katok2}  Let $F$ be a skew product, $\mu$ an $F$-invariant measure such that $\pi_*\mu=\nu$ and $\widehat{\mu}$ is a lift of $\mu$. Given $\widehat{\varepsilon}>0$ and $\gamma>0$, for $\nu$-a.e $x\in M$, $$\lim_{\varepsilon\to 0}\limsup_{n\to \infty} \frac{1}{n} \log r^{\widehat{F}}_x(n, \varepsilon, \widehat{\varepsilon}, \widehat{\mu}_x, \gamma)\leq \bar{h}_\mu(F\mid \pi).$$
\end{coro}

\section{Reparametrizations and entropy}
Reparametrizations and Yomdin theory play a key role in analyzing fiber entropy, as they enable us to bound it by controlling the number of reparametrizations. This approach works because fiber entropy can be characterized as the exponential growth rate of the iterations of a fiber-wise unstable manifold. Let us define the relevant objects.

\begin{defn}[Lift of a curve]
Let $r\geq 2$. For a fixed $x \in M$, we define a $C^{r}$ \textit{regular fiber curve} as the map
$$
\begin{aligned}
\sigma_x : [0,1] & \to M \times N, \\
t & \mapsto (x, \sigma(t)),
\end{aligned}
$$
where $\sigma : [0,1] \to N$ is a $C^{r}$ regular curve, meaning that $\sigma'(t) \neq 0$ for all $t \in [0,1]$.
Under this assumption, there exists a natural lift given by
\begin{gather*}
\widehat{\sigma}_x : [0,1] \;\longrightarrow\;M\times\widehat{N}, \\
\qquad t  \mapsto \bigl(x, \widehat{\sigma}(t)\bigr).
\end{gather*}
where $\widehat{\sigma}(t) = (\sigma(t), \mathbb{R}\sigma'(t))$.
\end{defn}

We define the $C^{r}$ norm of $\sigma_x$ as
$$
\|\sigma_x\|_{C^{r}} := \|\sigma\|_{C^{r}}
$$
Furthermore, we say that $\|\sigma_x\|_{C^{r}}<(\varepsilon,\widehat{\varepsilon})$ if
$$
\|\sigma\|_{C^{r}}<\varepsilon \text { and }\|\widehat{\sigma}\|_{C^{r}}<\widehat{\varepsilon} .
$$
\begin{defn}
Consider a $C^{r}$ curve $\sigma_x : [0,1] \to M\times N$. A reparametrization of $\sigma_x$ is any affine map $\psi : [0,1] \to [0,1]$ that is not constant.

Given a subset $J \subset [0,1]$, a family of reparametrizations of  $\sigma_x$ over $J$ is a collection $\mathcal{R}$ of such reparametrizations for which
$$
J \subset \bigcup_{\psi \in \mathcal{R}} \psi([0,1]).
$$
\end{defn}

We now explore how certain families of curve reparametrizations behave under repeated application of a skew product transformation. The central question is whether their geometric complexity, measured in the $C^{r}$ topology, stays controlled after several iterations. First, we choose a finite time scale $L$, a pair of small parameters $\varepsilon$ and $\widehat{\varepsilon}$, and a reference set $J \subset [0,1]$.

\begin{defn}
Let $r\geq 2$. Consider a skew product $F\colon M\times N \to M\times N$, a point $x \in M$ and a $C^{r}$ curve $\sigma_x$. A reparametrization $\psi$ of $\sigma_x$ is called $(C^{r}, F, L, \varepsilon, \widehat{\varepsilon})$-admissible up to time $n$ if there exists a finite, increasing sequence of times
$$
0 = n_0 < n_1 < \cdots < n_{\ell} = n
$$

such that the following two conditions are satisfied:
\begin{enumerate}[label = \arabic*.]
    \item \textbf{Step size bound:} For every $1 \leq j \leq \ell$, the time increment satisfies
    $$
   n_j - n_{j-1} \leq L.
   $$
    \item \textbf{Controlled image size:} For each $0 \leq j \leq \ell$, the image of the \\ reparametrized curve under the dynamics,
   $$
   F^{n_j} \circ \sigma_x \circ \psi,
   $$

   has $C^{r}$-size strictly bounded by $(\varepsilon, \widehat{\varepsilon})$.
   \end{enumerate}
\end{defn}

We refer to the integers $n_j$ as \textit{admissible times}. 
\begin{defn}
A collection $\mathcal{R}$ of reparametrizations of $\sigma_x$ over the set $J$ is said to be $(C^{r}, F, L, \varepsilon, \widehat{\varepsilon})$-admissible up to time $n$ if every map $\psi \in \mathcal{R}$ satisfies the $(C^{r}, F, L, \varepsilon, \widehat{\varepsilon})$-admissibility condition up to time $n$. 
\end{defn}

Recall that given a subset $\widehat{Z} \subset \widehat{N}$, we denote by
$$
r^{\widehat{F}}_x(n, \varepsilon, \widehat{\varepsilon}, \widehat{Z})
$$ the $(n, \varepsilon, \widehat{\varepsilon})$-covering number, which is the minimal number of $x$-fibered $(n, \varepsilon, \widehat{\varepsilon})$-Bowen balls of order $n$ needed to cover $\widehat{Z}$.


It turns out that, given an admissible family of reparametrizations, the minimal number of Bowen balls needed to cover the associated curve can be bounded by the same order of magnitude of the reparametrizations.

Let $\|D \widehat{F}\|_{\text {sup }}:=\sup_{x \in M} \sup_{\widehat{y} \in \widehat{N}} \|d_{\widehat{y}}\widehat{g}_x\|$.

\begin{lemma}\label{L1}
Let $r\geq 2$, $F$ be a skew product, $\sigma_x$ be a regular $C^{r}$-curve for some $x \in M$ and $J \subset [0,1]$. Fix two positive parameters $\varepsilon_*, \widehat{\varepsilon}_* > 0$, and an integer $L \geq 1$. Suppose $\mathcal{R}$ is a family of reparametrizations of $\sigma_x$ over $J$ which is $(C^{r}, F, L, \varepsilon_*, \widehat{\varepsilon}_*)$-admissible up to time $n$.

Then, for any $\varepsilon, \widehat{\varepsilon} > 0$, we conclude: $$
r_x^{\widehat{F}}(n, \varepsilon, \widehat{\varepsilon}, \widehat{\sigma}(J)) \leq \frac{2 \widehat{\varepsilon}_* \|D \widehat{F}\|_{\mathrm{sup}}^L}{\min(\varepsilon, \widehat{\varepsilon})} |\mathcal{R}|.
$$
\end{lemma}
\begin{proof}
The proof follows analogously to Lemma 4.12 of \cite{BCS}.
\end{proof}
The connection between curves, reparametrizations, and entropy is established by the following theorem.
As shown in Section \ref{sec: ENTROPY FORMULAS}, fiber entropy can be characterized by the exponential growth rate of the iterates of the fiber unstable manifold.
Given that, in the present setting, the fiber unstable manifolds are curves, we obtain the following relation:
\begin{theo}\label{entropy}
Fix $r\geq 2$. Let $F$ be a skew product and let $\mu$ be an ergodic, saddle type invariant measure with a system of conditional measures $\{ (\mu_x)_y^{\eta_x} \}$ supported on local fiber unstable manifolds.

Then, for $\mu$-a.e. $(x,y)$, and for any choice of:
\begin{itemize}
    \item a $C^{\infty}$ regular parameterized curve $\sigma: [0,1] \to \widetilde{W}^u(x, y)$,
    \item and a subset $J \subset [0,1]$ such that $(\mu_x)_{y}^{\eta_x}(\sigma(J)) > 0$,
\end{itemize}
we obtain the following entropy estimate.

Suppose $\mathcal{R}_n$ is a sequence of families of reparametrizations of the curve $\sigma$ over $J$, each of which is $(C^{r}, F, L, \varepsilon_*, \widehat{\varepsilon}_*)$-admissible up to time $n$, for some (arbitrary but fixed) constants $\varepsilon_*, \widehat{\varepsilon}_* > 0$ and an integer $L \geq 1$, independent of $n$. Then, the fiber entropy of $\mu$ satisfies:
$$
h_\mu(F \mid \pi) \leq \liminf_{n \to \infty} \frac{1}{n} \log |\mathcal{R}_n|.
$$
\end{theo}
\begin{proof}
    Fix constants $\varepsilon_*, \widehat{\varepsilon}_*, L > 0$ and let $\mathcal{R}_n$ be families of admissible reparametrizations as given in the statement. Let $\widehat{\sigma} : [0,1] \to \widehat{N}$ denote the canonical lift of the regular curve $\sigma$. Fix arbitrarily small $\varepsilon, \widehat{\varepsilon} > 0$. By Lemma \ref{L1}, we know:
$$
r_x^{\widehat{F}}(n, \varepsilon, \widehat{\varepsilon}, \widehat{\sigma}(J)) \leq \frac{2 \widehat{\varepsilon}_* \|D \widehat{F}\|_{\mathrm{sup}}^L}{\min(\varepsilon, \widehat{\varepsilon})} \cdot |\mathcal{R}_n|.
$$
Let us denote the constant on the right-hand side by
$$
C := C\left(F, \varepsilon, \widehat{\varepsilon}, \widehat{\varepsilon}_*, L\right) := \frac{2 \widehat{\varepsilon}_* \|D \widehat{F}\|_{\mathrm{sup}}^L}{\min(\varepsilon, \widehat{\varepsilon})},
$$
which is independent of $n$.
By the definition of $r_x^{\widehat{F}}(n, \varepsilon, \widehat{\varepsilon}, \widehat{\sigma}(J))$, there exist points $\widehat{y}_1, \ldots, \widehat{y}_\ell \in \widehat{N}$, with $\ell \leq C |\mathcal{R}_n|$, such that
$$
\widehat{\sigma}(J) \subset \bigcup_{i=1}^\ell B_x^{\widehat{F}}(\widehat{y}_i, n, \varepsilon, \widehat{\varepsilon}).
$$
Recall that $\widehat{\pi}_{\widehat{N}}: \widehat{N} \to N$ is given by $\widehat{\pi}_{\widehat{N}}(y, E) = y$. Projecting to the base, it follows that
$$
\sigma(J) \subset \bigcup_{i=1}^\ell B_x^F(\widehat{\pi}_{\widehat{N}}(\widehat{y}_i), n, \varepsilon),
$$
and therefore,
$$
r_x^F(n, \varepsilon, \sigma(J)) \leq r_x^{\widehat{F}}(n, \varepsilon, \widehat{\varepsilon}, \widehat{\sigma}(J)) \leq C |\mathcal{R}_n|.
$$
Applying Proposition \ref{P3}, we obtain:
$$
h_\mu(F \mid \pi) \leq \lim_{\varepsilon \to 0} \liminf_{n \to \infty} \frac{1}{n} \log r_x^F(n, \varepsilon, \sigma(J)) \leq \liminf_{n \to \infty} \frac{1}{n} \log |\mathcal{R}_n|.
$$
\end{proof}

\subsection{Reparametrization results}
As in the work of \cite{BCS}, one of the main tools we employ is Yomdin theory. We present here the statement of a version that closely follows theirs. In particular, it is possible to prove the existence of admissible reparametrizations whose cardinality can be controlled by the minimal number of Bowen balls needed to cover the curve. The proof of the proposition below is an easy adaptation of Corollary 4.14 of \cite{BCS}.

Let $r\geq 2$ and consider the following quantities:
$$
Q_r(F):=\max \left(\sup_{x \in M} \|g_x\|_{C^{r}},\; \sup_{x \in M} \|\widehat{g}_x\|_{C^{r}}\right),\quad Q_{r,L}(F):=\max _{j=1,2, \ldots, L} Q_r\left(F^j\right).
$$

Recall that $\|D \widehat{F}\|_{\text {sup }}=\sup_{x \in M} \sup_{\widehat{y} \in \widehat{N}} \|d_{\widehat{y}}\widehat{g}_x\|$.

\begin{prop}\label{yomdin} Let $F$ be a skew product. For every $2\leq r<\infty$ and $Q > 0$, there exist constants $\Upsilon=\Upsilon(r) > 0$ and $\varepsilon_Y = \varepsilon_Y(r,Q) > 0$ with the following property:  Given $x \in M$, suppose
\begin{enumerate}[label = (\arabic*)]
    \item $Q_{r, L}(F) < Q$,
    \item  The curve $\sigma_x : [0,1] \to M \times N$ is regular and has $C^{r}$-size bounded by $(\varepsilon, \widehat{\varepsilon})$, where $0 < \varepsilon, \widehat{\varepsilon} < \varepsilon_Y$,
    \item $L, n \geq 1$, $J \subset [0,1]$.
\end{enumerate}
Under these assumptions, there exists a collection $\mathcal{R}$ of reparametrizations of $\sigma$ over $J$, which is $(C^{r}, F, L, \varepsilon, \widehat{\varepsilon})$-admissible up to time $n$, and whose cardinality satisfies the estimate:
$$
|\mathcal{R}| \leq  C(r,F) \cdot r_{f(x)}^{\widehat{F}}\bigl(n, \varepsilon, \widehat{\varepsilon}, \pi_2(\widehat{F} \circ \widehat{\sigma}_x(J))\bigr),
$$ where $$C(r,F)=\Upsilon(r)^{\left\lceil \frac{n}{L}\right\rceil}\|D\widehat{F}^L\|^{\frac{\left\lfloor \frac{n}{L}\right\rfloor}{r-1}}_{\sup}\|D\widehat{F}^{n-L\lfloor \frac{n}{L}\rfloor}\|^{\frac{1}{r-1}}_{\sup}$$ and $\pi_2\colon M\times \widehat{N}\to \widehat{N}$ is the projection onto the second coordinate. 
\end{prop}

Observe that the constant $C(r,F)$ depends on the regularity $r$ that has been fixed. Although the map is acting $C^{\infty}$ in the fibers, since we do not know how the constant $\Upsilon(r)$ behaves with $r$, we are not able at this stage to simplify this quantity considering $r\to \infty$. 

In the following, we present a consequence of Proposition \ref{yomdin}. The next result generalizes to the present setting Proposition 5.1 of \cite{BCS}. We remark that the main difficulty appears when we apply Corollary \ref{katok2} since now all the parameters depends on the point $x\in M$. 

Define $\lambda(\widehat{F}):=\lim_{n\to \infty} \frac{1}{n}\log \|D\widehat{F}^n\|_{\sup}$.

\begin{prop}\label{5.1}
Let $F$ be a skew product, $\mu$ an $F$-invariant measure such that $\pi_*\mu=\nu$ and $\widehat{\mu}$ is a lift of $\mu$. Consider $2\leq r <\infty$ and real numbers $Q, \rho, \gamma > 0$. Let $\varepsilon_Y$ be given by Proposition \ref{yomdin}. Then, there exists $L_1 > 0$, $\varepsilon_1>0$ and $n_1>0$ such that for any $L \geq L_1$, $0 < \varepsilon<\varepsilon_1$, $0<\widehat{\varepsilon} < \varepsilon_Y$ and $n \geq n_1$, there exist
\begin{enumerate}[label=\emph{(\alph*)}]
    \item a neighborhood $\mathcal{U}_1$ of $F$,
    \item an open set $\widehat{U}_{1} \subset M\times \widehat{N}$ with $\widehat{\mu}(\widehat{U}_{1})>1-\gamma^2$ and $\widehat{\mu}(\partial \widehat{U}_{1})=0$,
such that the following property holds:
\end{enumerate}
For every skew product $G \in \mathcal{U}_1$ satisfying $Q_{r,L}(G) < Q$, and for every regular curve $\sigma_x$ with $C^{r}$-size bounded by $(\varepsilon, \widehat{\varepsilon})$, there exists a family $\mathcal{R}$ of reparametrizations over $\widehat{\sigma}_x^{-1}(\widehat{U}_{1})$ such that
\begin{enumerate}
    \item[$\bullet$] $\mathcal{R}$ is $\left(C^{r}, G, L, \varepsilon, \widehat{\varepsilon}\right)$-admissible up to time $n$,
    \item[$\bullet$] $|\mathcal{R}| \leq \exp \left[n\left(\bar{h}_{\mu}(F\mid\pi)+\frac{\lambda(\widehat{F})}{r-1}+\rho\right)\right]$.
\end{enumerate}
\end{prop}
\begin{proof}
We consider a compact set $K_{\gamma}\subset M$ such that the disintegration $x\mapsto \widehat{\mu}_x$ is continuous on $K_{\gamma}$, $\nu(K_{\gamma})>1-\gamma^2$ and Corollary \ref{katok2} holds for every $x\in K_{\gamma}$.

Let $\Upsilon(r)$ and $\varepsilon_Y$ be given by Proposition \ref{yomdin} and $L_1>0$ such that for every $L\geq L_1$, \begin{equation}\label{Y}\Upsilon(r)<\exp \left(\frac{\rho L}{10}\right),\end{equation} and $$\frac{1}{L}\log\|D\widehat{F}^L\|_{\sup}< \lambda(\widehat{F})+ \frac{\rho}{10}.$$ Fix $\widehat{\varepsilon}$ in $\left(0, \varepsilon_Y\right)$ and $L\geq L_1$. 

By Corollary \ref{katok2}, for every $x\in K_{\gamma}$, there exists $0<\bar{\varepsilon}(x)<\varepsilon_Y$ and $\bar{n}(x)>L$ such that for any $\varepsilon< \bar{\varepsilon}(x)$ and $n \geq \bar{n}(x)$,
\begin{equation}\label{E18}
\frac{1}{n} \log r^{\widehat{F}}_x\left(n, \frac{\varepsilon}{4}, \frac{\widehat{\varepsilon}}{4}, \widehat{\mu}_x, 1-\gamma^2\right)<\bar{h}_{\mu}(F\mid \pi)+\frac{\rho}{4}.
\end{equation}

Moreover, there exist $0<\varepsilon_1<\varepsilon_Y$, $n_1> L$ and a compact set $K'_{\gamma}\subset K_{\gamma}$ with $\nu(K'_{\gamma})>1-\gamma^2$ such that for every $x\in K'_{\gamma}$, $\bar{\varepsilon}(x)\geq \varepsilon_1$ and $\bar{n}_1(x)\leq n_1$. We can assume that $n_1$ was chosen such that for every $n\geq n_1$, \begin{equation}\label{ref1} \log \|D\widehat{F}^L\|_{\sup}<\frac{n\rho}{10}.
\end{equation}

Fix $\varepsilon<\varepsilon_1$ and $n\geq n_1$. By Equation (\ref{E18}), for every $x\in K'_{\gamma}$, there exists a set $\widehat{Z}_x \subset \widehat{N}$ with $\widehat{\mu}_x(\widehat{Z}_x) > 1 - \gamma^2$ such that
\begin{equation}\label{setZ}
r^{\widehat{F}}_x\left(n, \frac{\varepsilon}{4}, \frac{\widehat{\varepsilon}}{4}, \widehat{Z}_x\right) < \exp\left(n \, \bar{h}(F \mid \pi) + \frac{\rho}{4} \right).
\end{equation}

By the regularity of $\widehat{\mu}_x$, there exists a compact set $\widehat{K}_x\subset \widehat{Z}_x$ such that $\widehat{\mu}_x(\widehat{K}_x) > 1 - \gamma^2$ and $\widehat{K}_x$ also satisfies Equation (\ref{setZ}). 

Therefore, for every $x\in K'_{\gamma}$, there is a neighborhood $\widehat{U}_{1,x}$ of $\widehat{K}_x$, such that $\widehat{U}_{1,x}$ is contained in the union of a collection $\mathcal{C}_x$ of $x$-fibered $\left(n, \frac{\varepsilon}{4}, \frac{\widehat{\varepsilon}}{4}\right)$-Bowen balls for $\widehat{F}$ with cardinality at $\operatorname{most} \exp \left(n\left(\bar{h}(F\mid \pi)+\frac{\rho}{4}\right)\right)$. 

Since $x\mapsto \widehat{\mu}_x$ is continuous on $K'_{\gamma}$ and $K'_{\gamma}$ is compact, there exists $x_1, \cdots, x_l\in K'_{\gamma}$ and $\delta_i>0$ such that $K'_{\gamma}\subset \cup_{i=1}^l B(x_i, \delta_i)$ and if $y\in B(x_i,\delta_i)$, then $\widehat{\mu}_y(\widehat{U}_{1,x_i}) > 1 - \gamma^2$ and $\widehat{U}_{1,x_i}$ is contained in the union of a collection $\mathcal{C}_y$ of $y$-fibered $\left(n, \frac{\varepsilon}{2}, \frac{\widehat{\varepsilon}}{2}\right)$-Bowen balls for $\widehat{F}$ with cardinality at $\operatorname{most} \exp \left(n\left(\bar{h}(F\mid \pi)+\frac{\rho}{2}\right)\right)$. 

We define the following subset 
\begin{equation*}
\begin{aligned}
\widehat{U}_1=B(x_1,\delta_1)\times \widehat{U}_{1,x_1}\cup & \left(B(x_2,\delta_2)\setminus \overline{B(x_1, \delta_1)}\right)\times \widehat{U}_{1,x_2}\cdots\\
&\cup \left(B(x_l,\delta_l) \setminus \cap_{i=1}^{l-1}\overline{B(x_i, \delta_i)}\right)\times \widehat{U}_{1,x_l}.
\end{aligned}
\end{equation*}

Observe that $\widehat{U}_1$ is an open subset of $M\times \widehat{N}$ and we can suppose that $\widehat{\mu}\left(\partial \widehat{U}_{1}\right)=0$ and $\widehat{\mu}(\widehat{U}_1)>1-\gamma^2$.

For every $x\in M$ and $\sigma_x$ regular curve, we know that $$\widehat{F}(\sigma_x([0,1])\cap \widehat{U}_1)\subset \{f(x)\}\times g_x(\widehat{U}_{1,x_i}), \quad \text{for some } i=1,...,l.$$ This implies that there exists $\mathcal{U}_1$, a neighborhood of $F$ in the space of skew products, such that for every $G\in \mathcal{U}_1$, $\widehat{G}(\sigma_x([0,1])\cap \widehat{U}_1)$ can be covered using fibered $\left(n, \varepsilon,\widehat{\varepsilon}\right)$-Bowen balls for $\widehat{G}$ with cardinality at $\operatorname{most} \exp \left(n\left(\bar{h}(F\mid \pi)+\frac{\rho}{4}\right)\right)$. 

Moreover, we suppose that for every $G\in \mathcal{U}_1$, \begin{equation}\label{ref2}  \|D\widehat{G}^L\|_{\sup}\leq e^{\frac{nL}{10}}\|D\widehat{F}^L\|_{\sup}\quad \text{and} \quad \|D\widehat{G}\|_{\sup}\leq e^{\frac{\rho}{10}}\|D\widehat{F}\|_{\sup}
\end{equation}

If $G\in \mathcal{U}_1$ we conclude,
\begin{equation}\label{rG}
r_{g(x)}^{\widehat{G}}\left(n, \varepsilon, \widehat{\varepsilon}, \widehat{g}_x^G\left(\widehat{\sigma}[0,1] \cap \widehat{U}_{1}\right)\right) \leq \exp \left(n\left(\bar{h}(F\mid \pi)+\frac{\rho}{4}\right)\right) .
\end{equation}

Since $\varepsilon, \widehat{\varepsilon}<\varepsilon_Y$, if $Q_L(G)<Q$, we can apply Corollary \ref{yomdin} to obtain a family of reparametrizations which are admissible up to time $n$ and satisfy,
$$
|\mathcal{R}| \leq C(r,G)\; r_{g(x)}^{\widehat{G}}\left(n, \varepsilon, \widehat{\varepsilon}, \widehat{g}_x^G \left(\widehat{\sigma}[0,1] \cap \widehat{U}_{1}\right)\right).
$$

Using (\ref{Y}), (\ref{ref1}), (\ref{ref2}) and (\ref{rG}),

$$|\mathcal{R}|\leq\exp \left[n\left(\bar{h}(F\mid \pi)+\frac{\lambda(\widehat{F})}{r-1}+\rho\right)\right] .$$ 

\end{proof}

\section{Proof of Theorem \ref{thm-main}}

In this section we conclude the proof of the main theorem.

\subsection*{Proof of Theorem \ref{thm-main}} Let $(F_k, \mu_k)$ and $(F,\mu)$ satisfy the hypotheses of the theorem. That is,
\begin{enumerate}[label = -]
    \item the limits $\lim_k \lambda_1^c\left(F_k, \mu_k\right)$ and $\lim _k h_{\mu_k}\left(F_k\mid \pi\right)$ exist and are positive,
    \item $F_k$ converges to a skew product $F$,
    \item ${\mu}_k \xrightarrow{w^*} {\mu}$ for some ${F}$-invariant measure ${\mu}$.
\end{enumerate}

For $k$ big enough $h_{\mu_k}\left(F_k\mid \pi\right)>0$, then applying Theorem \ref{T2}, we can suppose that $\mu_k$ are of saddle type. 

Recall that when the measure $\mu_k$ is of saddle type, we can consider a partition $\eta^k$ subordinated to the fiber-wise unstable manifolds and a system of conditional measures $(\mu_{k,x})_{y}^{\eta^k_{x}}$ as in Section 5.1.
By Proposition \ref{expo}, there exists $\beta \in(0,1]$ and two $F$-invariant measures $\mu_0$ and $\mu_1$ such that $\mu=(1-\beta)\mu_0+\beta \mu_1$.


\begin{prop}\label{final} Let $r\leq 2 <\infty$. For any $\gamma > 0$, $\rho > 0$, $\varepsilon>0$ and $\widehat{\varepsilon}>0$ sufficiently small and $L$ big enough, there exists $k_* > 0$ such that for all $k \geq k_*$, there exist:
\begin{itemize}
    \item a set $A_k\subset M\times N$ such that $\mu_k(A_k)>0$ and
    \item an integer $N_k \in \mathbb{N}$ with the following property: 
\end{itemize}  

For every $(x_k,y_k) \in A_k$, the set $\widetilde{W}_k(x_k, y_k)$ and the measure $(\mu_{k,x_k})_{y_k}^{\eta^k_{x_k}}$ are well-defined. Moreover, for any $C^{\infty}$ regular curve $\sigma_k: [0,1] \to \widetilde{W}_k^u(x_k, y_k)$ satisfying $\|\sigma_k\|_{C^{\infty}} < (\varepsilon, \widehat{\varepsilon})$, any measurable set $J_k \subset [0,1]$ for which
$$
(\mu_{k,x_k})_{y_k}^{\eta^k_{x_k}}(\sigma_k(J_k) \cap A_k) > 0,
$$
and for all $n \geq N_k$, there exists a family of reparametrizations $\mathcal{R}_n^k$ of $\sigma_k$ over $J_k$ which is $(C^{r}, F_k, L, \varepsilon, \widehat{\varepsilon})$-admissible up to time $n$ and whose cardinality satisfies
$$
|\mathcal{R}_n^k| \leq \exp \left[n\left(\beta h_{\mu_1}(F \mid \pi) + \frac{\lambda(\widehat{F})}{r-1}+ O(\rho) + O(\gamma) \right) \right].
$$

\end{prop}
\begin{proof}
    The proof follows the same arguments than Section 7.4 of \cite{BCS}. The main results needed to conclude this proposition are item (iii) of Proposition \ref{P1}, Corollary \ref{yomdin} and Proposition \ref{5.1}. We also need to adapt Proposition 5.3 of \cite{BCS}, but it follows a reasoning analogous to Proposition \ref{5.1}. 
\end{proof}

For $k$ big enough, let $A_k$ be given by the Proposition above. Then, there exists $(x_k, y_k)\in A_k$ such that 
\begin{itemize}
   \item $(x_k,y_k)$ satisfies Theorem \ref{entropy}, 
   \item $y_k$ belongs to the support of the restriction of $(\mu_{k,x_k})_{y_k}^{\eta^k_{x_k}}$ to $A_k$.
\end{itemize}

Let $\sigma_k\colon [0,1]\to \widetilde{W}_k(x_k, y_k)$ be a $C^{\infty}$ regular curve that parametrizes a neighborhood of $y_k$ and has $C^{\infty}$ size smaller that $(\varepsilon, \widehat{\varepsilon})$. Considering $J_k=\sigma_k^{-1}(A_k)$ and applying Theorem \ref{entropy} and Proposition \ref{final} we conclude, $$h_{\mu_k}(F_k \mid \pi)\leq  \beta h_{\mu_1}(F \mid \pi) + \frac{\lambda(\widehat{F})}{r-1}+O(\rho) + O(\gamma).$$ Passing to the limits $k\to \infty$ and then $\gamma\to 0$ and $\rho\to 0$, we conclude, $$\lim_{k\to \infty} h_{\mu_k}(F_k \mid \pi)\leq  \beta h_{\mu_1}(F \mid \pi) + \frac{\lambda(\widehat{F})}{r-1}.$$ Observe that since the map $F$ acts $C^{\infty}$ on the fibers, the inequality holds for arbitrary large $r$. Therefore, the second term goes to zero and the first part of the theorem follows.

 Now let us assume that $\lambda_2(F,x,y)\leq0$ for $\mu$-a.e. $(x,y)$.
 As $\beta>0$, $\mu_1$-a.e. $(x,y)$ has $\lambda_2(F,x,y)\leq0$, so by Proposition~\ref{expo}
\begin{gather*}
\lim _{k \rightarrow \infty} \lambda_1^c\left(F_k, \mu_k\right)=\beta\, \lambda_1^c(F, \mu_1).
\end{gather*}

\qed

\section{Proof of Theorems \ref{thm-anosov} and \ref{thm-smooth-2}}


In order to conclude the theorems, we will use a result from \cite{BCS-SPR} that relates SPR to a condition on the Lyapunov exponents. Recall Definition \ref{spr}. Let $\tilde{f}:\tilde{M}\to \tilde{M}$ be a diffeomorphism on a compact $d$ dimensional manifold and $\mu$ be a invariant measure. Let $\lambda_1(\tilde{f},x)\geq \lambda_2(\tilde{f},x)\geq \dots \geq \lambda_d(\tilde{f}, x)$ be the Lyapunov exponents of $\tilde{f}$ and $\mu$.  Define $\Lambda^+(\mu)=\int \sum_{\lambda_i>0} \lambda_i(\tilde{f},x)d\mu$. 

We recall two conditions on $\tilde{f}$,
\begin{itemize}
\item[(EH)] \textbf{Entropy Hyperbolicity:} there exists $\chi>0$ with the following property. For every sequence of ergodic measures $\mu_k\to \mu$, if $h_{\mu_k}(\tilde{f})\to h_{top}(\tilde{f})$ then there exists $i$ such that $\lambda_i(x)>\chi>-\chi>\lambda_{i+1}(x)$ for $\mu$-a.e. $x$.
\item[(EC)] \textbf{Entropy Continuity of $\Lambda^+$:} For every sequence of ergodic measures $\mu_k\to \mu$ such that $h_{\mu_k}(\tilde{f})\to h_{top}(\tilde{f})$, then $\Lambda^+(\mu_k)\to \Lambda^+(\mu)$.
\end{itemize}

\begin{theo}[Theorem 3.1 of \cite{BCS-SPR}]
If $\tilde{f}$ is entropy hyperbolic and $\Lambda^+$ is entropy continuous, then $\tilde{f}$ is SPR.
\end{theo}

Therefore, in order to conclude Theorems \ref{thm-anosov} and \ref{thm-smooth-2}, we need to prove that in both cases $F$ is entropy hyperbolic and $\Lambda^+$ is entropy continuous.

We first prove the following proposition.
\begin{prop}\label{prop-beta-1}
Let $F$ be a skew product over $f$ such that $h_{top}(F)>h_{top}(f)$ and $\nu\mapsto h_\nu(f)$ is upper semi-continuous, then for every sequence of ergodic measures $\mu_k$ such that $\mu_k\to \mu$ and $h_{\mu_k}(F)\to h_{top}(F)$,
\begin{itemize}
\item $h_\mu(F)=h_{top}(F)$,
\item $\lambda^c_1(F,\mu_k)\to \lambda^c_1(F,\mu)>0$, and
\item $\lim_{k\to \infty}h_{\pi_*\mu_k}(f)=h_{\pi_*\mu}(f)$.
\end{itemize}
\end{prop}
\begin{proof}
Since $h_{top}(F)>h_{top}(f)$ for every ergodic measure $\nu$ with large entropy $h_{\nu}(F\mid\pi)>0$ and then, by Theorem~\ref{T2}, $\lambda^c_1(F,\nu)>0$.

By Theorem~\ref{thm-ab-rok} we know $$
h_{\mu_k}(F)=h_{\pi_*\mu_k}(f)+h_{\mu_k}(F\mid\pi).
$$

Taking $k\to \infty$, by Theorem~\ref{thm-main} and the semi-continuity of the metric entropy of $f$, we conclude that there exists $0<\beta\leq 1$, $\mu_0$ and $\mu_1$ such that $\mu=(1-\beta)\mu_0+\beta \mu_1$, and $$
h_{top}(F)\leq h_{\pi_* \mu}(f)+\beta h_{\mu_1}(F\mid \pi).
$$

Since $\beta h_{\mu_1}(F\mid\pi)+(1-\beta) h_{\mu_0}(F\mid\pi)=h_{\mu}(F\mid \pi)$, the equation above implies that $\beta=1$. We conclude that $h_\mu(F)=h_{top}(F)$. In particular, every ergodic component of $\mu$ is a measure of maximal entropy. Applying the observation at the beginning of the proof to the ergodic components of $\mu$, we conclude that $\lambda_2(F,x,y)\leq0$ $\mu$-almost everywhere. Therefore, again by Theorem \ref{thm-main}, the continuity of the Lyapunov exponents follows. 


Now assume by contradiction that $\lim_{k\to \infty}h_{\pi_*\mu_k}(f)$ is not $h_{\pi_*\mu}(f)$, by the upper semi-continuity, $\limsup_{k\to \infty}h_{\pi_*\mu_k}(f)\leq h_{\pi_*\mu}(f)$, then there exists a sub-sequence $k_j$ such that $\lim_{j\to \infty}h_{\pi_*\mu_{k_j}}(f)$ exists and is smaller than $h_{\pi_* \mu}(f)$.

This implies that 
$$
\begin{aligned}
h_{top}(F)\leq &  \lim_{j\to \infty}h_{\pi_*\mu_{k_j}}(f)+h_{\mu}(F\mid\pi)\\
<& h_{\pi_* \mu}(f)+h_{\mu}(F\mid\pi)\\
\leq & h_{top}(F),
\end{aligned}
$$
which is a contradiction.
\end{proof}

\def\Enc{E_i^{Nc}}
\def\Ef{E_{f,i}}

Let $F\colon M\times N\to M\times N$ be a skew product $C^1$ diffeomorphism over $f$. Suppose that $M$ is a $d$ dimensional manifold. 

Given an $F$ invariant measure $\mu$, we call $\lambda^{Nc}_i(F,x,y)$ and $\Enc(x,y)$, $i=1,\dots,d$ the Lyapunov exponents counted with multiplicity and the corresponding Oseledets sub-spaces of $F$ such that $\Enc(x,y)\not\subset   T_y N_x$.

Also we denote by $\lambda_i(f,x)$ and $\Ef(x)$ the Lyapunov exponents and corresponding Oseledets sub-spaces of $f$ at $x\in M$.

\begin{lemma}\label{lema-exponents}
Let $F$ be a skew product $C^1$ diffeomorphism over $f$ and $\mu$ an $F$ invariant measure, then for every $i=1,\dots,d$
\begin{itemize}
\item $\lambda^{Nc}_i(F,x,y)=\lambda_i(f,x)$, and
\item $D\pi(x,y)\Enc(x,y)=\Ef(x)$ for $\mu$ almost every $(x,y)$.
\end{itemize}
 
\end{lemma}
\begin{proof}
Recall that we denote by $\lambda_1^c(F,x,y)\geq \lambda_2^c(F,x,y)$ the Lyapunov exponents corresponding to the fiber.  
Take a $\mu$ total measure set of $M\times N$ such that the Oseledets theorem is valid for $F$ and every $(x,y)$, and also for every $f$ and every $x$.

Take $(x,y)$ in this set, fix $1\leq i\leq d$ and let $v^i\in \Enc(x,y)$. Observe that 
$$
\begin{aligned}
Df(x)D\pi(x,y)\Enc(x,y)&=D\pi(F(x,y))DF(x,y)\Enc(x,y)\\
&=D\pi(F(x,y))\Enc(F(x,y)),
\end{aligned}
$$ so $D\pi(x,y)\Enc(x,y)$ is $Df$ invariant. 

Now let us assume that $\lambda_i^{Nc}(F,x,y)\neq \lambda_j^c(F,x,y)$, $j=1,2$, then we know that $\Enc(x,y)$ does not intersect $T_y N_x$. Observe that $D\pi(x,y)$ preserves the dimension of $\Enc(x,y)$ and $\|D\pi(x,y)^{-1}\|$ is proportional to the angle between $\Enc(x,y)$ and $T_y N_x$, as $T_y N_x=E_1(x,y)\oplus E_2(x,y)$, by Oseledets Theorem it decrease sub exponentially along the orbit of $(x,y)$. 

Then we have 
$$
\begin{aligned}
\lambda_i^{Nc}(F,x,y)&=\lim \frac{1}{n}\log\| DF^n(x,y) v^i\|\\
&\geq \lim \frac{1}{n}\log\| D\pi(x,y) DF^n(x,y) v^i\|\\
& =\lim \frac{1}{n}\log\| Df^n(x)D\pi(x,y) v^i\|
\end{aligned}
$$ and also $$
\begin{aligned}
\lambda_i^{Nc}(F,x,y)&\leq \lim \left(\frac{1}{n}\log\| D\pi(F^n(x,y)) DF^n(x,y) v^i\|+\frac{1}{n}\log\| D\pi(F^n(x,y))^{-1} \| \right)\\
& =\lim\frac{1}{n}\log\| Df^n(x)D\pi(x,y) v^i\|+\lim \frac{1}{n}\log\| D\pi(F^n(x,y))^{-1} \|, 
\end{aligned}
$$
So we conclude that 
$\lim \frac{1}{n}\log\| Df^n(x)D\pi(x,y) v^i\|=\lambda_i^{Nc}(F,x,y)$. 

Now take the case that $\lambda^{Nc}_i(F,x,y)=\lambda^c_1(F,x,y)$, this implies that $\Enc(x,y)\supset E_1(x,y)$ properly. 
As $D\pi(x,y)\Enc(x,y)$ is invariant, it has an Oseledets sub decomposition. In particular, we can consider a vector $v\in D\pi(x,y)\Enc(x,y)$ with the property that  $\lim_{n\to\pm\infty}\frac{1}{n}\log \|Df^n(x)v\|= \lambda(f,x)$. Assume $\lambda(f,x)>\lambda_1(F,x,y)$ and take $V\in \Enc(x,y)$ with $D\pi(x,y)V=v$.

Then, $D\pi DF^n(x,y)V=Df^n(x)v$ growth exponentially with exponent $\lambda(f,x)$, a contradiction. If $\lambda(f,x)<\lambda_1(F,x,y)$ taking the inverse, we also get a contradiction. Therefore, we conclude that $\lambda(f,x)=\lambda^{Nc}_i(F,x,y)$. Analogously if 
$\lambda^{Nc}_i(F,x,y)=\lambda^c_2(F,x,y)$.

By the uniqueness of the Oseledets decomposition the result follows.
\end{proof}

\begin{proof}[Proof of Theorem~\ref{thm-smooth-2}]
Consider a sequence of ergodic measures $\mu_k$ such that $\mu_k\to \mu$ and $h_{\mu_k}(F)\to h_{top}(F)$. 
By the semi-continuity of the entropy for $C^{\infty}$ diffeomorphisms we know that $\mu$ is an MME and the ergodic decomposition of $\mu$ is composed by MMEs $\mu_e$. As $h_{top}(F)>h_{top}(f)$, by Theorem~\ref{T2}, $\lambda^c_1(F,\mu_e)\geq h_{top}(F)-h_{top}(f)$ and $h_{top}(f)-h_{top}(F)\geq\lambda^c_2(F,\mu_e)$ for every $\mu_e$, then for $\mu$- a.e. $(x,y)$, we conclude, $$\lambda^c_1(F,\mu_e)\geq h_{top}(F)-h_{top}(f)>h_{top}(f)-h_{top}(F)\geq \lambda^c_2(F,\mu_e).$$

By \cite{LW}, $h_{\mu}(F\mid \pi)\leq h_{top}(F\mid\pi)$. Therefore, since $h_{top}(F)>h_{top}(F\mid\pi)$, for $\pi_*\mu$ almost every $x$, $$\lambda_1(f,x)\geq h_{top}(F)-h_{top}(F\mid\pi)>h_{top}(F\mid\pi)-h_{top}(F)\geq \lambda_2(f,x).$$ Then, Lemma~\ref{lema-exponents} implies that $F$ is entropy hyperbolic. 

By Proposition~\ref{prop-beta-1}, we know that $h_{\pi_*\mu_k}(f)\to h_{\pi_*\mu}(f)$. By \cite{BCS}, this implies that there exist $0<\beta\leq 1$, $\nu_0$ and $\nu_1$ such that $\pi_{*}\mu=\beta\nu_1+(1-\beta)\nu_0$, $h_{\pi_*\mu}(f)\leq\beta h_{\nu_1}(f)$ and $$\lambda_1(f,\pi_*\mu_k)\to \int \lambda_1(f,x)d\nu_1.$$ Since $h_{top}(F)>h_{top}(F\mid\pi)$, we conclude that $h_{\nu_0}(f)>0$, therefore $\beta=1$ and $$\lambda_1(f,\pi_*\mu_k)\to \int \lambda_1(f,x)d\pi_*\mu.$$

Observe that $$
\Lambda^+(\mu_k)=\lambda^c_1(F,\mu_k)+\lambda_1^{Nc}(F,\mu_k).
$$
So by Proposition~\ref{prop-beta-1} and Lemma~\ref{lema-exponents}, we conclude that $\Lambda^+$ is entropy continuous.

\end{proof}

\begin{proof}[Proof of Theorem~\ref{thm-anosov}]
Observe that as $f$ is Anosov, by Lemma~\ref{lema-exponents}, $|\lambda_i^{Nc}(F,x,y)|$ are bounded away from zero, and the integral of the sum of the positive ones varies continuously with respect to the invariant measure. Moreover, by Proposition~\ref{prop-beta-1} there exists a measure of maximal entropy.

As in the previous proof we can conclude that $F$ is entropy hyperbolic and $\Lambda^+$ is entropy continuous. 

\end{proof}

{\em{Acknowledgments.}} KM and FV thanks the hospitality of Universidade Federal do Cear\'a where the main part of this work was elaborated.

\end{document}